\documentclass[a4paper,11pt]{article}
\usepackage{amsmath, amsfonts, amscd, amssymb, amsthm, enumerate, xypic}

\newcommand{\id}{\operatorname{id}}

\newcommand{\SL}{\operatorname{SL}}

\newcommand{\End}{\operatorname{End}}

\newcommand{\Sp}{\operatorname{Sp}}

\def\F{\mathbb{F}}

\def\R{\mathbb{R}}
\def\C{\mathbb{C}}

\def\H{\mathbb{H}}

\def\calA{\mathcal{A}}

\def\calU{\mathcal{U}}

\def\lcro{\mathopen{[\![}}
\def\rcro{\mathclose{]\!]}}

\theoremstyle{definition}

\theoremstyle{plain}
\newtheorem{theo}{Theorem}
\newtheorem{prop}[theo]{Proposition}

\newtheorem{lemma}[theo]{Lemma}

\theoremstyle{plain}

\theoremstyle{remark}

\title{A note on bireflectional elements of an algebra}
\author{Cl\'ement de Seguins Pazzis\footnote{Universit\'e de Versailles Saint-Quentin-en-Yvelines, Laboratoire de Math\'ematiques
de Versailles, 45 avenue des Etats-Unis, 78035 Versailles cedex, France}
\footnote{e-mail address: dsp.prof@gmail.com}}

\begin{document}

\thispagestyle{plain}

\maketitle

\begin{abstract}
A classical theorem of Wonenburger, Djokovi\v c, Hoffmann and Paige \cite{Djokovic,HoffmanPaige,Wonenburger} states that an element of the general linear group of a finite-dimensional vector space is the product of two involutions if and only if it is similar to its inverse.
We give a very elementary proof of this result when the underlying field $\F$ is algebraically closed with characteristic other than $2$. In that situation, the
result is generalized to the group of invertibles of any finite-dimensional algebra over $\F$.
\end{abstract}




\section{Bireflectional elements in an algebra over a field}

An element $g$ of a group $G$ is called bireflectional whenever $g=ab$ for involutions $a$ and $b$ of $G$, i.e.\ elements that satisfy $a^2=b^2=1_G$.
In that case, $g^{-1}=ba=bgb^{-1}$, so $g^{-1}$ is conjugated to $g$, and better still it is conjugated to $g$ through an involution.
And conversely if there exists an involution $b \in G$ such that $g^{-1}=bgb^{-1}$, then $gb$ is an involution and $g=(gb)b$ is bireflectional.

It is not true in general that an element of $G$ that is conjugated to its inverse is bireflectional. A standard example is the special linear group $\SL_2(\C)$,
in which every element is conjugated to its inverse although the only bireflectional elements are $\pm I_2$.
Another interesting example is the group of non-zero quaternions $\H^*$.
Two quaternions are conjugated if and only if they have the same norm and trace. It easily follows that every quaternion of norm $1$
is conjugated to its inverse in $\H^*$. Yet the only involutions in the group $\H^*$ are $\pm 1$, and hence they are also the only bireflectional elements.

The systematic study of bireflectional elements in classical groups was initiated by an article of M.J. Wonenburger \cite{Wonenburger}, who
proved that, in the general linear group of a finite-dimensional vector space over a field, the bireflectional elements are the ones that are
similar to their inverse (she actually discarded fields of characteristic $2$, an unnecessary restriction that was later lifted by Djokovi\v c \cite{Djokovic} and Hoffman and Paige
\cite{HoffmanPaige}). All the known proofs of this result require a deep understanding of the similarity classes in algebras of endomorphisms, and rely on the Frobenius canonical form.

The present note originated in an attempt to give a purely algebraic proof of Wonenburger's result.
Let us be clearer: say that we have invertible elements $x,y$ of a finite-dimensional $\F$-algebra $\calA$
(associative, with unity) such that $yxy^{-1}=x^{-1}$. If possible, we want to compute an element $z \in \calA^\times$ such that $zxz^{-1}=x^{-1}$ and
$z^2=1_\calA$. Alas, the example of the $\R$-algebra of quaternions shows that this is not possible in general.

Thus, we naturally restrict our scope to algebraically closed fields, and from now on we assume that $\F$ is so. And we also assume that
$\F$ does not have characteristic $2$. In that situation, the proof is very simple. Let us come back to the elements $x$ and $y$, and note
that for every \emph{odd} integer $k$, we have $y^k x y^{-k}=x^{-1}$, leading to $y^kx=x^{-1} y^k$.
Hence $zx=x^{-1}z$ for all $z \in y\F[y^2]$, and to conclude it will suffice to prove that $y\F[y^2]$ contains an involutory element.
The proof is surprisingly simple and is based upon the following classical result:

\begin{lemma}\label{lemma:squareroot}
Let $\calA$ be a finite-dimensional algebra over an algebraically closed field of characteristic other than $2$, and $z$ be an invertible element of $\calA$.
Then $z$ has a square root in $\calA$.
\end{lemma}

For the sake of completeness, we prove this lemma.

\begin{proof}
This is a classical application of Hensel's method. The element $z$ is annihilated by a non-zero polynomial, and we can
even take such a polynomial of the form $s^n$ where $s=\prod_{k=1}^m (t-a_k) \in \F[t]$ is split with simple roots in $\F$, all of them nonzero.
For each $k \in \lcro 1,m\rcro$, we choose a square root $b_k$ of $a_k$, and we take $r \in \F[t]$ as an arbitrary polynomial such that
$r(a_k)=b_k$ for all $k \in \lcro 1,m\rcro$. Note that $s$ divides $r^2-t$.

Now, assume that for some $k \geq 1$ we have found a polynomial $r$ such that $s^k$ divides $r^2-t$. Write $r^2-t=s^k v$ mod $s^{k+1}$ for some $v \in \F[t]$.
Letting $u \in \F[t]$, we note that $(r+s^k u)^2-t=s^k v+2s^k ru$ mod $s^{k+1}$.
Noting that $r$ is invertible mod $s$ (because so is $t$), and using the assumption that $\F$ does not have characteristic $2$, we see that $u$ can be adjusted so that
$s$ divides $v+2ru$. Hence, by induction, for all $k \geq 1$ we can find $r_k \in \F[t]$ such that $s^k$ divides $r_k^2-t$. The case $k=n$ yields the claimed result.
\end{proof}

Applying this lemma in the algebra $\F[y^2]$, we obtain $y^{-2}=p(y^2)^2$ for some $p \in \F[t]$, and hence $(yp(y^2))^2=1_\calA$.
This yields the claimed result, and we conclude as follows:

\begin{theo}\label{theo:bireflectional}
Let $\calA$ be a finite-dimensional algebra over an algebraically closed field of characteristic other than $2$.
Then the bireflectional elements of the group $(\calA^\times,\times)$ are the elements that are conjugated to their inverse.
\end{theo}

\section{Sums of square-zero elements in an algebra}

The previous method can be applied in different situations as well. In \cite{WuWang},
Wang and Wu examined the endomorphisms of a finite-dimensional complex vector space that split into the sum of two square-zero endomorphisms.
Their result was generalized by Botha \cite{Botha} to all fields. One of their key results is that
every endomorphism that is conjugated to its opposite is the sum of two square-zero endomorphisms. We generalize this result as follows:

\begin{theo}\label{theo:squarezero}
Let $\calA$ be a finite-dimensional algebra over an algebraically closed field $\F$ of characteristic other than $2$.
Let $x \in \calA$ be conjugated to $-x$. Then $x$ is the sum of two square-zero elements of $\calA$.
\end{theo}

\begin{proof}
Starting from an invertible $y \in \calA$ such that $yxy^{-1}=-x$, we find that $y^k x=-x y^k$ for every odd integer $k$.
Proceeding as before, we deduce that there is an involution $i \in \calA$ such that $ix=-xi$, that is $ixi=-x$.
Setting $p:=\frac{1}{2}(1_\calA+i)$ and $q:=\frac{1}{2}(1_\calA-i)$, we have $p+q=1_\calA$, $i=p-q$ and $pq=qp=0$. The equality $ixi=-x$
then leads to $x=(pxq+qxp)-(pxp+qxq)$, whereas $x=1_\calA x 1_\calA$ leads to $x=(pxq+qxp)+(pxp+qxq)$.
Taking the half sum yields $x=a+b$ with $a:=pxq$ and $b:=qxp$, which are obviously of square zero.
\end{proof}

In Theorem \ref{theo:squarezero}, the converse statement holds true for invertible elements in any ring $R$.
Indeed, let $x$ be an invertible element of a ring $R$, and $a,b$ be square-zero elements of $R$ such that $x=a+b$.
Then for $y:=ab-ba$ we see that $y^2=(ab)^2+(ba)^2=x^4$ and hence $y$ is invertible. Besides $yx=aba-bab=-xy$.
The invertibility of $x$ is unavoidable, even in finite-dimensional algebras over algebraically closed fields. Consider indeed the
algebra $\calA:=\C I_3+\mathrm{NT}_3(\C)$ of all upper-triangular $3 \times 3$ matrices with complex entries and equal diagonal entries.
We have the following decomposition
$$\underbrace{\begin{bmatrix}
0 & 1 & 0 \\
0 & 0 & 1 \\
0 & 0 & 0
\end{bmatrix}}_J=\underbrace{\begin{bmatrix}
0 & 1 & 0 \\
0 & 0 & 0 \\
0 & 0 & 0
\end{bmatrix}}_A+\underbrace{\begin{bmatrix}
0 & 0 & 0 \\
0 & 0 & 1 \\
0 & 0 & 0
\end{bmatrix}}_B$$
where $A^2=B^2=0$. Yet $J$ is not conjugated to $-J$ in $\calA$, for the conjugate of any matrix $M$ of $\calA$ must have the same entries
as $M$ on the first super-diagonal.

\section{Bireflectional elements in an algebra with involution}

We finish by coming back to the initial problem of bireflectional elements, this time in the context of an $\F$-algebra $\calA$ equipped
with an $\F$-linear involution $x \mapsto x^\star$. It is natural to ask
what becomes of Theorem \ref{theo:bireflectional} when the setting is the unitary group
$$\mathcal{U}(\calA):=\{x \in \calA : \; x^\star x=1_\calA\}.$$
However, it is known that the result fails with this level of generality, e.g.\ when $\calA$ is the algebra of all endomorphisms of a complex vector space with finite dimension $2n>0$
equipped with a symplectic form $s$, with respect to which the involution is the standard adjunction. In that case, it is known that every unitary element of the symplectic group $\Sp(s)$
is conjugated to its inverse, but there are elements of $\Sp(s)$ that are not bireflectional (this is trivial if $n=1$ because in that case the only involutions in $\Sp(s)$ are $\pm \id$).
Here is the best general result we could come up with:

\begin{prop}
Let $(\calA,x \mapsto x^\star)$ be a finite-dimensional algebra with linear involution over an algebraically closed field $\F$ with characteristic other than $2$.
Let $x \in \calU(\calA)$ be conjugated to its inverse in $\mathcal{A}$. Then there exists $y \in \calU(\calA)$ such that $yxy^{-1}=x^{-1}$ and $y^4=1_\calA$.
\end{prop}

\begin{proof}
The first step consists in noting that $x$ is conjugated to its inverse in the group $\mathcal{U}(\mathcal{A})$.
This is the classical part that uses the polar decomposition trick, and we recall it for the sake of completeness.
Start with $z \in \calA^\times$ such that $zxz^{-1}=x^{-1}$. Taking the adjoint, we also find $z^\star x^{-1} (z^\star)^{-1}=x$
and we deduce that $z^\star z$ commutes with $x$. Taking a square root $s$ of $z^\star z$ in $\F[z^\star z]$, we still have that $s$ commutes
with $x$ and now $z':=zs^{-1}$ is unitary and satisfies $z'x(z')^{-1}=x^{-1}$.

So, we can assume that $z \in \mathcal{U}(\mathcal{A})$. Next, there are two special situations to consider.
\begin{itemize}
\item \textbf{Special case 1: $z-z^{-1}$ is invertible.} Note that $(z-z^{-1})^2$ is invertible and that it belongs to $\F[(z+z^{-1})^2]$.
By Lemma \ref{lemma:squareroot}, we can find $p \in \F[t]$ such that $(z-z^{-1})^{-2}=p((z+z^{-1})^2)^2$.
Set $u:=i (z-z^{-1}) p((z+z^{-1})^2)$, where $i$ is an element of $\F$ such that $i^2=-1$.
Since $z^2$ is invertible and $\mathcal{A}$ is finite-dimensional, we have $\F[z^2]=\F[z^2,z^{-2}]$.
Therefore it is clear that $u \in z \F[z^2]$, to the effect that $uxu^{-1}=x^{-1}$. Finally, $u$ is skewselfadjoint and satisfies $u^2=-1_\calA$, whence $u \in \mathcal{U}(\calA)$.
\item \textbf{Special case 2: $z-z^{-1}$ is nilpotent.} Then $z+z^{-1}$ is invertible because
$(z+z^{-1})^2=4.1_\calA+(z-z^{-1})^2$. This time around, we find $q \in \F[t]$ such that $(z+z^{-1})^{-2}=q((z+z^{-1})^2)^2$,
and we set $u:=(z+z^{-1}) q((z+z^{-1})^2)$.
Clearly  $u \in z \F[z^2]$, whence $uxu^{-1}=x^{-1}$. Finally, $u$ is selfadjoint and satisfies $u^2=1_\calA$, whence $u \in \mathcal{U}(\calA)$.
\end{itemize}
To obtain the general case, we can simply use the Fitting decomposition of the selfadjoint element $a:=(z-z^{-1})^2$. This yields an idempotent $p \in \F[a]$
such that $pap$ is invertible in $p\calA p$ and $p'ap'$ (where $p':=1_\calA-p)$ is nilpotent in $p'\calA p'$.
Noting that $p$ is selfadjoint, we get that $p\calA p$ and $p'\calA p'$ are stable under the involution of $\calA$ under consideration.
Besides, we note that $(z-z^{-1})^2$ commutes with both $x$ and $z$, and hence $p$ commutes with both $x$ and $z$.
We deduce that $pzp$ conjugates $pxp$ into its inverse in $p \calA p$, and $p'zp'$ conjugates $p'xp'$ into its inverse in $p' \calA p'$.

Using the previous two special cases, we obtain a unitary $z_1 \in p\calA p$ such that $z_1$ conjugates $pxp$ into its inverse in $p \calA p$
and $(z_1)^2=-1_{p \calA p}$, and a unitary $z_2 \in p'\calA p'$ such that $z_2$ conjugates $p'xp'$ into its inverse in $p' \calA p'$
and $(z_2)^2=1_{p' \calA p'}$. Hence, $y:=z_1+z_2$ is unitary in $\calA$, and we have the algebraic relations $y xy^{-1}=x^{-1}$ and $y^4=1_\calA$.
\end{proof}

Remark that with the previous proof we can conclude that every element of $\calU(\calA)$ that is conjugated to its inverse splits into a product $yz$ with
$y,z$ in $\mathcal{U}(\calA)$ such that $y^4=z^4=1_\calA$.

Finally, note that if the algebra $\calA$ equals $\End(V)$ for some finite-dimensional vector space $V$ over $\F$, then
every unitary element of $\calA$ is conjugated to its adjoint, and hence to its inverse, inside of $\calA$.
This is a classical consequence of the Skolem-Noether theorem on automorphisms of central simple algebras, together with the fact that every endomorphism is similar to its transpose.

\end{document}